\documentclass[11pt]{amsart}

\usepackage{xypic,amsmath,amssymb}

\newtheorem{proposition}{Proposition}[section]
\newtheorem{lemma}[proposition]{Lemma}
\newtheorem{corollary}[proposition]{Corollary}
\newtheorem{theorem}[proposition]{Theorem}

\theoremstyle{definition}

\theoremstyle{remark}

\newcommand{\thlabel}[1]{\label{th:#1}}
\newcommand{\thref}[1]{Theorem~\ref{th:#1}}
\newcommand{\selabel}[1]{\label{se:#1}}
\newcommand{\seref}[1]{Section~\ref{se:#1}}
\newcommand{\lelabel}[1]{\label{le:#1}}
\newcommand{\leref}[1]{Lemma~\ref{le:#1}}
\newcommand{\prlabel}[1]{\label{pr:#1}}
\newcommand{\prref}[1]{Proposition~\ref{pr:#1}}
\newcommand{\colabel}[1]{\label{co:#1}}
\newcommand{\coref}[1]{Corollary~\ref{co:#1}}

\newcommand{\eqlabel}[1]{\label{eq:#1}}
\newcommand{\equref}[1]{(\ref{eq:#1})}

\def\mapright#1{\smash{\mathop{\longrightarrow}\limits^{#1}}}

\def\doublerightbis#1#2{{\lower.2ex\vbox{
\hbox{${\smash{\mathop{\longrightarrow}\limits^{#1}}}$}\vspace*{-4mm}
\hbox{${\smash{\mathop{\longrightarrow}\limits_{#2}}}$}}}}

\def\equal#1{\smash{\mathop{=}\limits^{#1}}}

\newcommand{\Hom}{{\rm Hom}}
\newcommand{\HOM}{{\rm HOM}}
\newcommand{\Ext}{{\rm Ext}}
\newcommand{\EXT}{{\rm EXT}}
\newcommand{\Tor}{{\rm Tor}}

\def\lan{\langle}
\def\ran{\rangle}
\def\ot{\otimes}

\newcommand{\ev}{{\rm ev}}

\newcommand{\Cc}{\mathcal{C}}

\newcommand{\Ee}{\mathcal{E}}

\newcommand{\Mm}{\mathcal{M}}

\def\*C{{}^*\hspace*{-1pt}{\Cc}}

\def\text#1{{\rm {\rm #1}}}
\def\smashco{\mathrel>\joinrel\mathrel\triangleleft}

\allowdisplaybreaks[4]

\begin{document}
\title[Cohomology of comodules over smash coproducts]{On the cohomology of comodules over smash coproducts}
\author{S. Caenepeel}
\address{Faculty of Engineering,
Vrije Universiteit Brussel, VUB, Pleinlaan 2, B-1050 Brussels, Belgium}
\email{scaenepe@vub.ac.be}
\urladdr{http://homepages.vub.ac.be/\~{}scaenepe/}
\author{T. Gu\'ed\'enon}
\address{D\'epartement de Math\'ematiques, Universit\'e de Ziguinchor, BP 523, Ziguinchor, Senegal}
\email{thomas.guedenon@univ-zig.sn}



\begin{abstract}
We consider the category of comodules over a smash coproduct coalgebra $C\smashco H$. We show that there is a Grothendieck spectral sequence 
connecting the derived functors of  the Hom functors coming from
$C\smashco H$-colinear, $H$-colinear and
rational $C$-colinear morphisms. We give several applications and connect our results
to existing spectral sequences in the literature.
\end{abstract}
\maketitle

\section*{Introduction}
The category of comodules over a coalgebra over a field is a Grothendieck category
with enough injectives,
so we can consider right derived functors of left exact functors from
the category of comodules to vector spaces. Examples of such functors include
colinear homomorphisms, rational homomorphisms and coinvariants. The aim of this
paper is to establish spectral sequences relating these derived functors. The classical
procedure to do this is due to Grothendieck \cite{Groth}, and allows to associate
a spectral sequence to a composition $G\circ F$ of two left exact functors, under the
assumption that $F$ preserves injective objects. This technique has been applied in
the literature in several situations: in \cite{Magid}, Magid considers an affine algebraic
group $H$ acting on a commutative noetherian algebra $R$ over an algebraically closed field,
and studies the cohomology of $R\cdot H$-modules; in \cite{Gued} the cohomology of
modules over a smash product algebra is studied; in \cite{CaenGued}, algebras with
a Hopf algebra coaction are considered, and then the cohomology of the associated
relative Hopf modules is investigated.\\
The setting of the present paper is as follows: take a Hopf algebra $H$ over a field,
and an $H$-comodule coalgebra $C$, that is a
coalgebra in the category of (right) $C$-comodules. Then we can consider the smash
coproduct coalgebra $C\smashco H$, and the right derived functors of
$\Hom^{C\smashco H}$. This functor is isomorphic to a composition of two
functors of type $\Hom^H$ and $\HOM^C$ (\prref{1.5}); the second functor preserves
injectives (\leref{2.2}), so that we can apply Grothendieck's results, leading to our
main result \thref{2.6}.\\
\seref{1} contains some preliminary results, mainly on comodules over the smash
coproduct. \seref{2} begins with properties of injective comodules, needed in order
to prove the main result \thref{2.6}. We then look at particular situations; for example,
the coinvariants functor $(-)^{{\rm co}H}$ is isomorpic to $\Hom^H(k,-)$, leading to
a spectral sequence involving the coinvariants functor, see \coref{2.7}. In the case
where $H$ is cosemisimple, the exact sequence collapses, leading to \coref{2.8}.
In \seref{3}, we present some applications. We present a new completely different
proof of a result in \cite{CaeDasRai} on the cosemisimplicity of the smash coproduct.
As another consequence, we obtain the Hochschild-Serre spectral sequence for smash coproducts of coalgebras and Hopf algebras, we refer to \cite{HochSerre}, where the sequence has been established for groups extensions, to \cite[Chap. XVI, Sec. 6]{CartEilen}, to \cite{Haboush, ClineParsh} for rational algebraic group actions, and to \cite{Guich} for smash products of  algebras and groups. Some other results on homological properties in the category of comodules over smash coproducts can be found in \cite{DasNastTorec} and \cite{Zhu}.\\
Our results can be applied to the following smash coproducts, see \seref{3} for detail:
\begin{enumerate}
\item $B \smashco H$, where $H$ is cosemisimple and $B=C^{{\rm co}H}$.
\item $C\smashco kG$, where $G$ is a group and $C$ is a $G$-graded coalgebra. Then the category of $C\smashco kG$-comodules is isomorphic to the category of $G$-graded $C$-comodules. 
\item $A \smashco k^G$, where $A$ is a Hopf algebra, $G$ is a finite group acting as Hopf algebra automorphisms on $A$ and $k^G$ is the algebra of set functions from $G$ to $k$, see \cite{MontgVegaWither}. Note that $k^G$ is a finite dimensional Hopf algebra. 
\item $A(G) \cong A(H) \smashco A(K)$, where $G$ is an affine algebraic group, semi-direct product of algebraic subgroups $H$ and $K$, with coordinate rings respectively denoted $A(G)$, $A(L)$ and $A(K)$, see \cite{Molnar}. Note that all these rings are commutative Hopf algebras, the rational $G$-modules and the $A(G)$-comodules coincide, and $A(L)$ is an $A(K)$-comodule coalgebra.
\end{enumerate}
 
\section{Preliminary results} \selabel{1}
\subsection{Comodules over a Hopf algebra}
Troughout this paper, we work over a field $k$.
All unlabelled $\ot$ and $\Hom$ are  over $k$.
Let $H$ be a Hopf algebra over $k$. The category $\Mm^H$ consisting of right $H$-comodules and
right $H$-colinear maps is monoidal. The tensor product of two right $H$-comodules $(M,\rho^M)$ and
$(N,\rho^N)$ is the usual tensor product $M\ot N$, with coaction
\begin{equation}\eqlabel{1.1}
\rho^{M\ot N}(m\ot n)=m_{[0]}\ot n_{[0]}\ot m_{[1]}n_{[1]}.
\end{equation}
Here we use the Sweedler notation $\rho(m)=m_{[0]}\ot m_{[1]}$ for the right $H$-coaction $\rho$ on $M$.
The subspace of $H$-coinvariants of $(M,\rho)\in \Mm^H$ is 
defined as
$$M^{{\rm co}H}=\{m\in M~|~\rho^H(m)=m\ot 1\}\cong \Hom^H(k,M).$$
It is known that $H$ is cosemisimple if and only if the coinvariants functor
$(-)^{{\rm co}H}:\ \Mm^H\to \Mm$ is exact. \\
For $M,N\in \Mm^{H}$, the canonical map
$$\iota:\ \Hom(M,N)\ot H\to \Hom(M,N\ot H),~~\iota(f\ot h)(m)=f(m)\ot h$$
is injective (we work over a field), and is considered as an inclusion. $f\in \Hom(M,N)$ is called $H$-rational if $\rho:\ \Hom(M,N)\to \Hom(M,N\ot H)$,
$$\rho(f)(m)=f(m_{[0]})_{[0]} \otimes f(m_{[0]})_{[1]}S(m_{[1]}),$$
factors through $\Hom(M,N)\ot H$. In this situation, we write
$\rho(f)=f_{[0]}\ot f_{[1]}\in \Hom(M,N)\ot H$.
$\rho(f)$ is then characterized by the property
\begin{equation}\eqlabel{1.2}
f_{[0]}(m)\ot f_{[1]}=f(m_{[0]})_{[0]} \otimes f(m_{[0]})_{[1]}S(m_{[1]}),
\end{equation}
or, equivalently,
\begin{equation}\eqlabel{1.2b}
\rho(f(m))=f_{[0]}(m_{[0]})\ot f_{[1]}m_{[1]},
\end{equation}
for all $m\in M$. $\HOM(M,N)$, the subspace of $\Hom(M,N)$ consisting of rational 
maps is a right $H$-comodule. If $M$ or $H$ is finite dimensional, then the canonical
inclusion $\iota$ is an isomorphism, and $\HOM(M,N)=\Hom(M,N)$. It follows from \equref{1.2b} that the coaction on
$\HOM(M,N)$ is designed in such a way that the evaluation map
$\ev:\ \HOM(M,N)\ot M\to N$
is right $H$-colinear. We remark that a different definition of rationality is used in \cite{CaenGued};
actually the characterizing property \equref{1.2} reduces to the one in \cite{CaenGued} if $H$ is
replaced by $H^{\rm op}$.

\subsection{Comodules over a smash coproduct}
A coalgebra in the category $\Mm^H$ is called a right $H$-comodule coalgebra. This is a coalgebra
$C$ with a right $H$-coaction
$$\rho:\ C\to C\ot H,~~\rho(c)=c_{[0]}\ot c_{[1]}$$
such that $\Delta_C$ and $\varepsilon_C$ are right $H$-colinear, that is,
\begin{equation}\eqlabel{1.1a}
\Delta_C(c_{[0]}) \otimes c_{[1]}= c_{1[0]} \otimes c_{2[0]} \otimes c_{1[1]}c_{2[1]}~~ \hbox{and} ~~ \epsilon_C(c_{[0]})c_{[1]}=\epsilon_C(c)1_H,
\end{equation}
for all $c\in C$. The smash coproduct $C \smashco H$ is the vector space $C\ot H$
with counit $\varepsilon_C \smashco \varepsilon_H$ and
comultiplication given by the formula
\begin{equation}\eqlabel{1.3}
\Delta(c \smashco h)=(c_1 \smashco c_{2[1]}h_2) \otimes (c_{2[0]} \smashco h_1),
\end{equation}
for all $c\in C$ and $h\in H$. A direct verification shows that $C \smashco H$ is a coalgebra,
so we can consider the category of right $C \smashco H$-comodules $\Mm^{C\smashco H}$.
$\Mm^{C\smashco H}$ has direct sums, and, since we are working over a field, it is
a Grothendieck category. We can provide the following alternative description of
$\Mm^{C\smashco H}$. Recall from \cite{CaeDasRai} that a right $(C , H)$-comodule
is a vector space $M$ together with a right $C$-coaction $\rho^C$ and a right $H$-coaction
$\rho^H$, with corresponding Sweedler notation
$$\rho^C(m)= m_{\{0\}} \otimes m_{\{1\}}\in M\ot C~~{\rm and}~~
\rho^H(m)=m_{[0]} \otimes m_{[1]}\in M\ot H,$$
satisfying the compatibility relation
\begin{equation}\eqlabel{star}
\rho^C(m_{[0]})\ot m_{[1]}=m_{\{0\}[0]} \otimes m_{\{1\}[0]} \otimes m_{\{0\}[1]}m_{\{1\}[1]},
\end{equation}
for all $m\in M$.
The category $\Mm^{C,H}$ of right $(C,H)$-comodules and right $C$-colinear and $H$-colinear
maps is isomorphic to $\Mm^{C\smashco H}$, see \cite[Prop. 1.3]{CaeDasRai}. We will now present
some properties of $\Mm^{C\smashco H}$.

\begin{proposition}\prlabel{1.1}
Let $C$ be a right $H$-comodule coalgebra, and take $M,N\in \Mm^{C,H}$.
$\HOM^C(M,N)=\Hom^C(M,N) \cap \HOM(M,N)$,
the subspace of $\HOM(M,N)$ consisting of right $C$-colinear maps, is an
$H$-subcomodule of $\HOM(M,N)$. This establishes a functor $\HOM^C(M,-):\
\Mm^{C,H}\to \Mm^H$. Furthermore
$$\Hom^{C\smashco H}(M,N)= \HOM^C(M,N)^{{\rm co}H}.$$
\end{proposition}

\begin{proof}
Take $f\in \HOM^C(M,N)$. 
Take a basis $\{h_i~|~i\in I\}$ of $H$ as a vector space. Then we can write
$$\rho(f)=f_{[0]}\ot f_{[1]}= \sum_i f_i\ot h_i,$$
where only a finite number of the $f_i$ are different from $0$.\\
For a fixed $j\in I$, we can write
\begin{equation}\eqlabel{1.1.0}
\Delta(h_j)=\sum_{i\in I} a_{ji}\ot h_i,
\end{equation}
for some $a_{ji}\in H$, with only finitely many of the $a_{ji}$ different from $0$.\\
For all $m\in M$, we have
\begin{eqnarray*}
&&\hspace*{-15mm}
\rho^C(f_{[0]}(m))\ot f_{[1]}
\equal{\equref{1.2}}
\rho^C(f(m_{[0]})_{[0]}) \otimes f(m_{[0]})_{[1]}S(m_{[1]})\\
&\equal{\equref{star}}&
f(m_{[0]})_{\{0\}[0]}  \otimes  f(m_{[0]})_{\{1\}[0]}  \otimes f(m_{[0]})_{\{0\}[1]}f(m_{[0]})_{\{1\}[1]}S(m_{[1]})    \\
&\equal{(*)}&
f(m_{[0]\{0\}})_{[0]} \otimes  m_{[0]\{1\}[0]} \otimes f(m_{[0]\{0\}})_{[1]}m_{[0]\{1\}[1]}S(m_{[1]})\\
&\equal{\equref{star}}&
f(m_{\{0\}[0]})_{[0]}  \otimes m_{\{1\}[0]}  \otimes f(m_{\{0\}[0]})_{[1]}m_{\{1\}[1]}S(m_{\{1\}[2]})S(m_{\{0\}[1]})\\
&=&
f(m_{\{0\}[0]})_{[0]}  \otimes m_{\{1\}}  \otimes f(m_{\{0\}[0]})_{[1]}S(m_{\{0\}[1]})\\
&\equal{\equref{1.2}}&
f_{[0]}(m_{\{0\}})\ot m_{\{1\}} \ot f_{[1]}.
\end{eqnarray*}
At $(*)$, we used the fact that $f$ is right $C$-colinear.
This can be rewritten
as
\begin{equation}\eqlabel{1.1.1}
\sum_i \rho^C(f_i(m))\ot h_{i}=\sum_i f_i(m_{\{0\}})\ot m_{\{1\}} \ot h_i \in
\bigoplus_{i\in I} N\ot H \ot kh_i.
\end{equation}
For $l\in I$, take the projection of \equref{1.1.1} onto $N\ot H \ot kh_l$. This
gives
$$\rho^C(f_l(m))=f_l(m_{\{0\}})\ot m_{\{1\}},$$
and therefore each $f_l$ is right $C$-colinear.\\
A straightforward (and well-known) computation shows that
$$\rho(f_{[0]})(m)\ot f_{[1]}=f_{[0]}(m)\ot \Delta(f_{[1]}).$$
This can be rewritten as
\begin{equation}\eqlabel{1.1.2}
\sum_{i\in I} \rho(f_i)(m)\ot h_i=\sum_{j\in I} f_j(m)\ot \Delta(h_j)
\equal{\equref{1.1.0}} \sum_{i\in I} \sum_{j\in I} f_j(m)\ot a_{ji}\ot h_i.
\end{equation}
For $l\in I$, take the projection of \equref{1.1.2} onto $N\ot H \ot kh_l$. This
gives
$$\rho(f_l)(m)=\sum_{j\in I} f_j(m)\ot a_{jl}.$$
It follows that each $f_l$ is rational. We now have shown that $f_l\in \HOM^C(M,N)$,
for all $l\in I$, and
$$\rho(f)=\sum_{l\in I} f_l\ot h_l\in \HOM^C(M,N)\ot H.$$
We leave it to the reader to show that this construction is functorial.
It is easy to show that $f\in \Hom^H(M,N)$ if and only if $\rho(f)=f\ot 1$, or, equivalently,
$f\in \HOM(M,N)^{{\rm co}H}$. Consequently $f\in \Hom^{C\smashco H}(M,N)$ if and only
if $f\in \Hom^C(M,N)\cap \HOM(M,N)^{{\rm co}H}=\HOM^C(M,N)^{{\rm co}H}$.
\end{proof}

\begin{lemma} \lelabel{1.2}
Let $M,N,P$ be $C \smashco H$-comodules. If $f \in  HOM^C(M, N)$ and $g \in HOM^C(N,P)$, then $g \circ f \in  HOM^C(N,P)$.
\end{lemma}

\begin{proof} 
It is well-known that $g \circ f$ is rational and $C$-colinear.
\end{proof}

\begin{lemma}\lelabel{1.3}
We have a functor $-\ot C:\ \Mm^{C\smashco H}\to \Mm^{C\smashco H}$.
\end{lemma}

\begin{proof}
Let $M$ be a right $(C,H)$-comodule. $C$ coacts on
$M\ot C$ as follows
\begin{equation}\eqlabel{1.3.1}
\rho^C(m\ot c)=m \otimes c_1 \otimes c_2.
\end{equation}
Let us verify that the compatibility condition \equref{star} holds.
\begin{eqnarray*}
&&\hspace*{-2cm}
\rho^C((m\ot c)_{[0]})\ot (m\ot c)_{[1]}=
 m_{[0]} \otimes c_{[0]1} \otimes c_{[0]2} \otimes  m_{[1]}c_{[1]} \\
&\equal{\equref{1.1a}}&m_{[0]} \otimes c_{1[0]} \otimes c_{2[0]} \otimes  m_{[1]}c_{1[1]}c_{2[1]} \\
&=&(m \otimes c_1)_{[0]} \otimes c_{2[0]} \otimes  (m \otimes c_1)_{[1]}c_{2[1]} \\
&=&(m \otimes c)_{\{0\}[0]} \otimes  (m \otimes c)_{\{1\}[0]} \otimes  (m \otimes c)_{\{0\}[1]} (m \otimes c)_{\{1\}[1]} .
\end{eqnarray*}
We leave it to the reader to show that this construction is functorial.
\end{proof}

\begin{lemma}\lelabel{1.4}
Let $L\in \Mm^H$ and $M\in \Mm^{C\smashco H}$. Then $L \otimes M$ is a right $C \smashco H$-comodule.
\end{lemma}

\begin{proof}
$C$ coacts on
$L\ot M$ as follows
\begin{equation}\eqlabel{1.4.1}
\rho^C(l\ot m)=(l \otimes m)_{\{0\}} \otimes (l \otimes m)_{\{1\}}=l \otimes m_{\{0\}} \otimes m_{\{1\}}.
\end{equation}
Let us verify that the compatibility condition \equref{star} holds.
\begin{eqnarray*}
&&\hspace*{-2cm}
\rho^C((l\ot m)_{[0]})\ot (l\ot m)_{[1]}=
l_{[0]} \otimes m_{[0]\{0\}} \otimes m_{[0]\{1\}} \otimes l_{[1]}m_{[1]} \\
&\equal{\equref{star}}& l_{[0]} \otimes m_{\{0\}[0]} \otimes m_{\{1\}[0]} \otimes l_{[1]} m_{\{0\}[1]} m_{\{1\}[1]} \\
&=&( l \otimes m_{\{0\}})_{[0]} \otimes m_{\{1\}[0]} \otimes (l \otimes m_{\{0\}})_{[1]}m_{\{1\}[1]} \\
&=&( l \otimes m)_{\{0\}[0]} \otimes (l \otimes m)_{\{1\}[0]} \otimes (l \otimes m)_{\{0\}[1]}(l \otimes m)_{\{1\}[1]}.
\end{eqnarray*}
\end{proof}

\begin{proposition}\prlabel{1.5}
Let $M,N\in \Mm^{C\smashco H}$ and $L\in \Mm^H$. Then we have an isomorphism
$$\Hom^{C\smashco H}(L\ot M,N)\cong \Hom^H(L, \HOM^C(M,N)).$$
\end{proposition}

\begin{proof}
We will first show that the canonical isomorphism
$$\phi:\ \Hom(L \otimes M, N) \to \Hom(L,\Hom(M,N)),~~\phi(f)(l)(m)= f(l \otimes m)$$
restricts and corestricts to a map
$$\phi:\ \Hom^{C\smashco H}(L\ot M,N)\to \Hom^H(L, \HOM^C(M,N)).$$
Take $f\in \Hom^{C\smashco H}(L\ot M,N)$. For all $l\in L$ and $m\in M$, we have that
\begin{eqnarray*}
&&\hspace*{-2cm}
\rho(\phi(f)(l))(m)=
(\phi(f)(l)(m_{[0]}))_ {[0]} \ot (\phi(f)(l)(m_{[0]}))_ {[1]}   S(m_{[1]}) \\
 &=&  f(l \ot m_{[0]})_ {[0]} \ot f(l \ot m_{[0]})_ {[1]}   S(m_{[1]}) \\
 &\equal{(*)}&  f((l \ot m_{[0]})_ {[0]}) \ot (l \ot m_{[0]})_ {[1]}  S(m_{[1]}) \\
 &=& f(l_{[0]} \ot m_{[0]}) \ot l_{[1]}m_{[1]}S(m_{[2]})\\
 &=& f(l_{[0]} \ot m)\ot l_{[1]}= \phi(f)(l_{[0]})(m)\ot l_{[1]}.
 \end{eqnarray*}
At $(*)$, we used the fact that $f$ is $H$-colinear. This proves that $\phi(f)(l)$
is rational, and, moreover, that
$$\rho(\phi(f)(l))=\phi(f)(l_{[0]})\ot l_{[1]}.$$
This  tells us that $\phi(f):\ L\to \HOM(M,N)$ is right $H$-colinear. It remains to
be shown that $\phi(f)(l):\ M\to N$ is right $C$-colinear, for every $l\in L$. Indeed, for all $m\in M$,
we have that
$$
\rho^C((\phi(f)(l))(m)=\rho^C(f(l\ot m))\equal{(*)}
f(l\ot m_{\{0\}})\ot m_{\{1\}}= (\phi(f)(l))(m_{\{0\}})\ot m_{\{1\}}.$$
At $(*)$, we used the fact that $f$ is right $C$-colinear.
The inverse $\phi^{-1}:\ \Hom(L,\Hom(M,N))\to \Hom(L \otimes M, N)$ of $\phi$ is given by the formula
$$\phi^{-1}(g)(l\ot m)=(g(l))(m).$$
The proof is finished if we can show that $\phi^{-1}$ restricts and corestricts to a map
$\phi^{-1}:\ \Hom^H(L, \HOM^C(M,N))\to \Hom^{C\smashco H}(L\ot M,N)$.
Take a right $H$-colinear map $g:\ L\to \HOM^C(M,N)$. For all $l\in L$ and $m\in M$, we
have that
\begin{eqnarray*}
&&\hspace*{-15mm}
\rho^C(\phi^{-1}(g)(l\ot m))=\rho^C((g(l))(m))\equal{(*)}g(l)(m_{\{0\}})\ot m_{\{1\}}\\
&=&\phi^{-1}(g)(l\ot m_{\{0\}})\ot m_{\{1\}}=
\phi^{-1}(g)((l\ot m)_{\{0\}})\ot (l\ot m)_{\{1\}},
\end{eqnarray*}
and it follows that $\phi^{-1}(g)$ is right $C$-colinear.
As $g:\ L\to \HOM(M,N)$ is right $H$-linear, we have that $\rho(g(l))=g(l_{[0]})\ot l_{[1]}$,
and
\begin{equation}\eqlabel{1.5.1}
g(l)_{[0]}(m)\ot g(l)_{[1]}=g(l_{[0]})(m)\ot l_{[1]},
\end{equation}
for all $m\in M$. Now
\begin{eqnarray*}
&&\hspace*{-15mm}
\phi^{-1}(g)((l\ot m)_{[0]})\ot (l\ot m)_{[1]}=
\phi^{-1}(g)(l_{[0]}\ot m_{[0]})\ot l_{[1]}m_{[1]}\\
&=& (g(l_{[0]}))( m_{[0]})\ot l_{[1]}m_{[1]}
\equal{\equref{1.5.1}}
g(l)_{[0]}( m_{[0]}) g(l)_{[1]}m_{[1]}\\
&\equal{\equref{1.2}}&
(g(l))(m_{[0]})_{[0]}\ot (g(l))(m_{[0]})_{[1]}S(m_{[1]})m_{2]}\\
&=&\rho^H((g(l))(m))=\rho^H(\phi^{-1}(g)(l\ot m)),
\end{eqnarray*}
and this proves that  $\phi^{-1}(g)$ is right $H$-colinear.
\end{proof}

\begin{lemma}\lelabel{1.6}
We make the following additional assumptions: $H$ is commutative and $C$ is
cocommutative. Furthermore $M$ is finite dimensional, or $H$ and $C$ are
finite dimensional. If $M,N\in {\mathcal M}^{C\smashco H}$, then $\Hom^C(M,N)
\in {\mathcal M}^{C\smashco H}$. 
\end{lemma}

\begin{proof}
The $H$-coaction is defined as in \prref{1.1}, taking into account that
$\Hom^C(M,N)=\HOM^C(M,N)$. The $C$-coaction is defined as follows:
for $f\in \Hom^C(M,N)$,
$\rho^C(f)=f_{\{0\}} \otimes f_{\{1\}}\in \Hom(M,N\ot C)\cong \Hom(M,N)\ot C$ if and only if
\begin{equation}\eqlabel{1.6.1}
f_{\{0\}}(m) \otimes f_{\{1\}}=f(m_{\{0\}}) \otimes m_{\{1\}},
\end{equation}
for all $m\in M$. Fix a basis $\{c_i~|~i\in I\}$ of $C$. There exist $f_i\in \Hom(M,N)$
such that $\rho^C(f)=\sum_i f_i\ot c_i$. For all $m\in M$, we have that
\begin{eqnarray*}
&&\hspace*{-15mm}
\sum_i \rho^C(f_i(m))\ot c_i=
f_{\{0\}}(m)_{\{0\}}\ot f_{\{0\}}(m)_{\{1\}}\ot f_{\{1\}}\\
&\equal{\equref{1.6.1}}&
f(m_{\{0\}})_{\{0\}} \ot f((m_{\{0\}})_{\{1\}}\ot m_{\{1\}}
\equal{(*)} f(m_{\{0\}})\ot m_{\{1\}}\ot m_{\{2\}}\\
&\equal{(**)}&
f_{\{0\}}(m_{\{0\}})\ot m_{\{1\}}\ot f_{\{1\}}=\sum_i f_i(m_{\{0\}})\ot m_{\{1\}}\ot c_i.
\end{eqnarray*}
At $(*)$, we used the $C$-colinearity of $f$; at $(**)$, we used the cocommutativity of $C$.
It follows that $\rho^C(f_i(m))=f_i(m_{\{0\}})\ot m_{\{1\}}$, for all $m\in M$ and
$i\in I$. We conclude that each $f_i$ is right $C$-colinear, and
$\rho^C(f)\in \Hom^C(M,N)\ot C$. Let us show that $\rho^C$ is coassociative. For all
$m\in M$, we have that
\begin{eqnarray*}
&&\hspace*{-15mm}
f_{\{0\}\{0\}}(m) \otimes f_{\{0\}\{1\}} \otimes f_{\{1\}}\equal{\equref{1.6.1}}
f_{\{0\}}(m_{\{0\}})\ot m_{\{1\}} \ot f_{\{1\}}\\
&\equal{\equref{1.6.1}}& f_{\{0\}}(m_{\{0\}})\ot m_{\{2\}} \ot m_{\{1\}}\equal{(*)}
f(m_{\{0\}}) \otimes \Delta_C(m_{\{1\}})\\
&\equal{\equref{1.6.1}}& f_{\{0\}}(m)\ot \Delta^C(f_{(1)}).
\end{eqnarray*}
At $(*)$, we used the cocommutativity of $C$. It follows that
$\rho^C(f_{\{0\}})\ot f_{\{1\}}=f_{\{0\}}\ot \Delta_C(f_{\{1\}})$, so that
$\rho^C$ is coassociative. It is obvious that 
$$\varepsilon_C(f_{\{1\}})f_{\{0\}}(m)=f(m_{\{0\}})\varepsilon_C(m_{\{1\}})=f(m),$$
and the counit property follows. Finally, we have to prove the compatibility condition
\equref{star}. For all $m\in M$, we have that
\begin{eqnarray*}
&&\hspace*{-15mm}
f_{\{0\}[0]}(m) \otimes  f_{\{1\}[0]} \otimes  f_{\{0\}[1]} f_{\{1\}[1]}\\
&\equal{\equref{1.2}}&
f_{\{0\}}(m_{[0]})_{[0]} \otimes  f_{\{1\}[0]} \otimes  f_{\{0\}}(m_{[0]})_{[1]}S(m_ {[1]}) f_{\{1\}[1]}\\
&\equal{\equref{1.6.1}}&
f(m_{[0]\{0\}})_{[0]} \otimes  m_{[0]\{1\}[0]} \otimes  f(m_{[0]\{0\}})_{[1]}S(m_ {[1]})  m_{[0]\{1\}[1]}\\
&\equal{\equref{star}}&
f(m_{\{0\}[0]})_{[0]} \otimes  m_{\{1\}[0]} \otimes  f(m_{\{0\}[0]})_{[1]}S(m_ {\{1\}[2]}) 
S(m_ {\{0\}[1]})m_{\{1\}[1]}\\
&\equal{(*)}&
f(m_{\{0\}[0]})_{[0]} \otimes  m_{\{1\}} \otimes  f(m_{\{0\}[0]})_{[1]} S(m_ {\{0\}[1]})\\
&\equal{\equref{1.2}}&
f_{[0]}(m_{\{0\}}) \otimes  m_{\{1\}} \otimes  f_{[1]}
= f_{[0]\{0\}}(m) \otimes  f_{[0]\{1\}} \otimes  f_{[1]},
\end{eqnarray*}
and the compatibility relation follows; At $(*)$, we used the fact that $H$ is commutative.
\end{proof}

\section{A spectral sequence}\selabel{2}
\begin{lemma}\lelabel{2.1}
Let $I$ be a right $(C,H)$-comodule. We have two left exact functors (resp.
covariant and contravariant)
$\HOM^C(I,-),~\HOM^C(-,I):\ \Mm^{C\smashco H}\to \Mm^H$.
\end{lemma}

\begin{proof}
In \prref{1.1} the two functors are defined at the level of objects.
It is easy to see that this construction is functorial, and that the two functors
are left exact.
 \end{proof}
 
It is well-known that the category of comodules over a coalgebra over a field is a
Grothendieck category. Therefore the categories $\Mm^C$, $\Mm^H$ and
$\Mm^{C\smashco H}$ are Grothendieck categories. For a right $(C,H)$-comodule $M$,
we will establish a spectral
sequence connecting the right derived functors of 
$\HOM^C(M , -)$, $\Hom^{C\smashco H}(M , -)$ and $(-)^{{\rm co}H}$.

\begin{lemma}\lelabel{2.2}
Take right $(C,H)$-comodules $M$ and $I$, and assume that $I$ is injective in
$\Mm^{C\smashco H}$. Then $\HOM^C(M,I)$ is an injective object in $\Mm^H$.
\end{lemma}

\begin{proof}
Consider an exact sequence $0\to L_1\to L_2\to L_3\to 0$ in $\Mm^H$, and, a fortiori,
in $\Mm_k$. The sequence $0\to L_1\ot M\to L_2\ot M\to L_3\ot M\to 0$ is exact in
$\Mm_k$; the spaces and maps in the sequence are in $\Mm^{C\smashco H}$, hence
we have an exact sequence in $\Mm^{C\smashco H}$. Now $I\in \Mm^{C\smashco H}$,
so 
\begin{eqnarray*}
0& \to & \Hom^{C\smashco H}( L_3 \otimes M, I) \to \Hom^{C\smashco H}(L_2 \otimes M ,I) \\
&\to&
 \Hom^{C\smashco H}(L_1 \otimes M, I) \to 0
 \end{eqnarray*}
is an exact sequence in $\Mm_k$. It follows from \prref{1.5} that this sequence is isomorphic
to
\begin{eqnarray*}
0& \to &\Hom^H( L_3, \HOM^C(M, I)) \to  \Hom^H( L_2, \HOM^C(M, I)) \\
&\to  &
\Hom^H( L_1, \HOM^C(M, I)) \to 0.
\end{eqnarray*}
This shows that the functor $\Hom^H( -, \HOM^C(M, I)):\ \Mm^H\to \Mm_k$ is exact,
proving the assertion.
\end{proof}

\begin{lemma}\lelabel{2.3}
If $I\in \Mm^{C\smashco H}$ is injective, then $\HOM^C(-, I):\ \Mm^{C\smashco H}\to \Mm^H$ is an exact functor.
\end{lemma}

\begin{proof}
We know from \leref{2.1} that $\HOM^C(-, I)$ is left exact. Let $i:\ M\to N$
be a monomorphism in $\Mm^{C\smashco H}$. We need to show that
$\HOM^C(i,I):\ \HOM^C(N,I)\to \HOM^C(M,I)$ is surjective.\\
Take $f\in \HOM^C(M,I)$, and let $V$ be a finite dimensional $H$-subcomodule of  $\HOM^C(M, I)$ containing $f$. By \leref{1.4}, the map
$V\ot i:\ V\ot M\to V\ot N$ is a monomorphism in $\Mm^{C\smashco H}$. We have seen that the evaluation map
$\ev:\ V\ot M\to I,~~F(v\ot m)=v(m)$
is right $H$-colinear. It follows from \leref{1.4} that $V\ot M\in \Mm^{C\smashco H}$. Now we claim that $\ev$ is also 
right $C$-colinear. Indeed, for all $v\in V$ and $m\in M$, we have that
$$\rho^C(\ev(v\ot m))= \rho^C(v(m))\equal{(*)} v(m_{\{0\}})\ot m_{\{1\}}=
(\ev\ot C)(\rho^C(v\ot m)).$$
At $(*)$, we used the right $C$-colinearity of $v$. It follows that $\ev$ is a morphism in $\Mm^{C\smashco H}$,
and the injectivity of $I\in \Mm^{C\smashco H}$ entails the existence of a $C\smashco H$-colinear map
$G:\ V\ot N\to I$ such that $G\circ (V\ot i)=\ev$. Consider the map $g:\ N\to I$,
$g(n)=G(f\ot n)$. For all $n\in N$, we have that
\begin{eqnarray*}
&&\hspace*{-2cm}
g(n_{[0]})_{[0]}\ot g(n_{[0]})_{[1]}S(n_{[1]})=
G(f\ot n_{[0]})_{[0]}\ot G(f\ot n_{[0]})_{[1]}S(n_{[1]})\\
&\equal{(*)}& 
G(f_{[0]}\ot n_{[0]})\ot f_{[1]} n_{[1]}S(n_{[2]})=G(f_{[0]}\ot n)\ot f_{[1]}.
\end{eqnarray*}
At $(*)$, we used the right $C$-colinearity of $G$.
This proves that $g$ is $H$-rational, and $\rho^H(g)= G(f_{[0]}\ot -)\ot f_{[1]}$.
Furthermore
$$
\rho^C(g(n))=\rho^C(G(f\ot n))\equal{(*)}G(f\ot n_{\{0\}})\ot n_{\{1\}}=g(n_{\{0\}})\ot n_{\{1\}},$$
proving that $g$ is right $C$-colinear. At $(*)$, we used the right $C$-colinearity of $G$.
We conclude that $g\in \HOM^C(N,I)$, and, finally
$\HOM^C(i,I)(g)=g\circ i=f$, since
$$f(m)=\ev(f\ot m)=(G\circ (V\ot i))(f\ot m)=(g\circ i)(m),$$
for all $m\in M$. This completes the proof.
\end{proof}

\begin{lemma}\lelabel{2.4}
\begin{enumerate}
\item If $I\in \Mm^{C\smashco H}$ is injective, then $I$ is also injective as
an object of $\Mm^C$.
\item Assume that $H$ is cosemisimple, and that $M\in \Mm^{C\smashco H}$
is finite dimensional and projective as a right $C$-comodule. Then $M$ is also projective as a right
$C\smashco H$-comodule.
\end{enumerate}
\end{lemma}

\begin{proof}
(1) We have seen in \leref{1.2} that $I\ot C\in \Mm^{C\smashco H}$. For all
$u\in I$, we have
$$\rho^H(\rho^C(u))=
u_{\{0\}[0]} \otimes u_{\{1\}[0]} \otimes u_{\{0\}[1]}u_{\{1\}[1]} 
\equal{\equref{star}}\rho^C(u_{[0]})\ot u_{[1]},$$
so that $\rho^C:\ I\to I\ot C$ is right $H$-colinear. It is well-known that
$\rho^C$ is a monomorphism in $\Mm^C$, and we can conclude that
$\rho^C$ is a monomorphism in $\Mm^{C\smashco H}$. Thus we have
a short exact sequence
$$0 \rightarrow I \mapright{\rho^C} I \otimes C \rightarrow (I \otimes C)/I \rightarrow 0$$
in $\Mm^{C\smashco H}$. Since $I\in \Mm^{C\smashco H}$ is injective,
this sequence splits in $\Mm^{C\smashco H}$, and, a fortiori, in $\Mm^C$.
Hence $I$ is a direct factor of $I \otimes C$ in $\Mm^C$. It is well-known that
$I \otimes C$ is injective as a $C$-comodule, and we conclude that $I$ is injective
as a $C$-comodule.\\
(2) Take an exact sequence
$$0 \rightarrow N_1  \rightarrow N_2 \rightarrow N_3 \rightarrow 0$$
in $\Mm^{C\smashco H}$, and, a fortiori, in $\Mm^C$. $M\in \Mm^C$ is projective,
so the sequence
$$0 \rightarrow \Hom^C(M, N_1)  \rightarrow \Hom^C(M,N_2) \rightarrow \Hom^C(M,N_3) \rightarrow 0$$
is exact in $\Mm^C$. The functor $(-)^{{\rm co}H}:\ \Mm^C\to \Mm_k$ is exact since
$H$ is cosemisimple, so we have an exact sequence
$$0 \rightarrow \Hom^C(M, N_1)^{{\rm co}H}  \rightarrow \Hom^C(M,N_2)^{{\rm co}H} \rightarrow \Hom^C(M,N_3)^{{\rm co}H} \rightarrow 0$$
in $\Mm_k$. By \prref{1.1}, this sequence can be rewritten as
$$0 \rightarrow \Hom^{C\smashco H}(M, N_1)  \rightarrow \Hom^{C\smashco H}(M,N_2) \rightarrow \Hom^{C\smashco H}(M,N_3) \rightarrow 0.$$
We conclude that $\Hom^{C\smashco H}(M,-)$ is an exact functor, as needed.
\end{proof}

For $M \in {\mathcal M}^C$, let $\Ext^C(M, -)$ be the right derived functors of the functor 
$\Hom^C(M, -):\ \Mm^C\to \Mm_k$. For $M,N \in \Mm^C$, $\Ext^{C,p}(M,N)=H^p(\Hom^C(M, \Ee^{\star}))$, the $p$-th cohomology group of the complex $\Hom^C(M, \Ee^{\star})$
associated to an injective resolution $\Ee^{\star}$ of $N$ in $\Mm^C$.\\
If $M$ is finite dimensional, then $M^*\in {}^C\Mm$ with left $C$-coaction
$$\lambda(m^*)=\sum_i \lan m^*,m_{i[0]}\ran m_{i[1]}\ot m_i^*,$$
for all $m^*\in M^*$, where $\sum_i m_i\ot m_i^*\in M\ot M^*$ is the finite dual basis of $M$.
The natural isomorphism $-\ot M^*\cong \Hom(M,-):\ \Mm^C\to \Mm_k$ restricts to
a natural isomorphism $-\square_C M^*\cong \Hom^C(M,-)$, where $\square_C$
is the cotensor product over $C$. This implies immediately that we have natural
isomorphism 
$\Tor^C(N,M^*)\cong \Ext^C(M,N)$ between the right derived functors of 
$-\square_C M^*$ and $\Hom^C(M,-)$.\\
Take $M\in \Mm^{C\smashco H}$, and let $\EXT^C(M,-)$ be the right derived functors
of the functor  $\HOM^C(M,-)$ from \prref{1.1}.

\begin{lemma}\lelabel{2.5}
Take $M,N \in \Mm^{C\smashco H}$, and assume that $M$ or $H$ is finite dimensional.
Then
$\EXT^C(M,N)= \Ext^C(M,N)$ as vector spaces.
\end{lemma}

\begin{proof}
Take an injective resolution $\Ee^\star$ of $N\in \Mm^{C\smashco H}$. It follows from
\leref{2.4} that $\Ee^\star$ is also an injective resolution of $N\in \Mm^{C}$.
Since $M$ or $H$ is finite dimensional, we have that
$\HOM^C(M,P)= \Hom^C(M,P)$ for all $P\in \Mm^{C\smashco H}$, and consequently
$$\EXT^{C}(M,N)=H^\bullet(\HOM^C(M,\Ee^*))=H^\bullet(\Hom^C(M,\Ee^*))=\Ext^{C}(M,N).$$
\end{proof}

\begin{theorem}\thlabel{2.6}
For $M,N\in \Mm^{C\smashco H}$ and $L\in \Mm^H$, we have a spectral sequence
\begin{equation}\eqlabel{2.6.1}
\Ext^{H,p}(L, \EXT^{C,q}(M,N)) \Rightarrow \Ext^{C\smashco H,p+q}(L \otimes M,N),
\end{equation}
for all $p,q\geq 0$.
\end{theorem}

\begin{proof}
We have left exact functors
$F=\HOM^C(M,-):\ \Mm^{C\smashco H}\to \Mm^H$ and $G= \Hom^H(L,-):\ \Mm^H\to \Mm_k$.
Now $G\circ F= \Hom^{C\smashco H}(L\ot M, N)$ (\prref{1.5})
and $F$ preserves injectives (\leref{2.2}), so we have the Grothendieck spectral
sequence for composite functors
$$(R^pG)((R^qF)(N)) \Rightarrow R^{p+q}(G \circ F)(N),$$
for $p,q\geq 0$, see \cite[Theorem 2.4.1]{Groth}, which specifies to \equref{2.6.1}.
\end{proof}

We have seen in \seref{1} that the coinvariants functor $(-)^{{\rm co}H}$ is
naturally isomorphic to $\Hom^H(k,-)$. This implies that the right derived functors
$R^p(-)^{{\rm co}H}$ are naturally isomorphic to $\Ext^H(k,-)$, and we obtain the
following corollary from \thref{2.6}.

\begin{corollary}\colabel{2.7}
For $M,N\in \Mm^{C\smashco H}$, we have a spectral sequence
$$R^p(\EXT^{C,q}(M,N))^{{\rm co}H} \Rightarrow \Ext^{C\smashco H,p+q}(M,N).$$
\end{corollary}

\begin{corollary}\colabel{2.8}
$M,N\in \Mm^{C\smashco H}$ and $L\in \Mm^H$. If $H$ is cosemisimple, then
we have
isomorphisms of vector spaces
\begin{eqnarray}
\Hom^H(L, \EXT^{C,q}(M,N)) &\cong & \Ext^{C\smashco H,q}(L \otimes M,N);\eqlabel{2.8.1}\\
\EXT^{C,q}(M,N)^{{\rm co}H} &\cong& \Ext^{C\smashco H,q }(M,N),\eqlabel{2.8.3}
\end{eqnarray}
for all $q\geq 0$.
\end{corollary}

\begin{proof}
The category ${\mathcal M}^H$ is completely  reducible if $H$ is cosemisimple,
and then the spectral sequence \equref{2.6.1} collapses, yielding 
\equref{2.8.1}.
\equref{2.8.3} follows after we take $L=k$ in \equref{2.8.1}.
\end{proof}

It follows from \leref{2.5} that we can replace $\EXT^C$ by $\Ext^C$
in \thref{2.6} and Corollaries \ref{co:2.7} and \ref{co:2.8}, in the situation where $M$ or $H$ is finite dimensional.

\begin{corollary} \colabel{2.9}
Assume that $H$ is cosemisimple and take $M\in {\mathcal M}^{C\smashco H}$.
\begin{enumerate}
\item If $M$ is projective as a $C$-comodule and $M$ or $H$ is finite dimensional, then 
$L \otimes M$ is projective in  $\Mm^{C\smashco H}$, for every $L\in \Mm^H$.
\item If $H$ is finite dimensional, then $M$ is injective in $\Mm^C$ if and only if
$M$ is injective in $\Mm^{C\smashco H}$.
\end{enumerate}
\end{corollary}

\begin{proof}
(1) Since $M$ is projective in $\Mm^C$,  $\EXT^{C,q}(M,N)=\Ext^{C,q}(M,N)=0$ for all $q>0$ and $N\in \Mm^{C\smashco H}$. It then follows from \equref{2.8.1}
that $\Ext^{C\smashco H,q}(L \otimes M,N)\cong \Hom^H(L, \EXT^{C,q}(M,N))=0$
for all $q>0$ and $N\in \Mm^{C\smashco H}$, and this implies that $L \otimes M$ is
projective as a right $C\smashco H$-comodule.\\
(2) If $M$ is injective in $\Mm^C$, then $\Ext^{C,1}(N,M)=0$ for all $N\in {\mathcal M}^{C\smashco H}$,
and it follows from \equref{2.8.1} (with $L=k$) that $\Ext^{C\smashco H,1}(N,M)=0$, for 
all $N\in {\mathcal M}^{C\smashco H}$, and $M$ is injective in $\Mm^{C\smashco H}$
The converse implication follows from \leref{2.4}.
\end{proof}

\section{Further applications}\selabel{3}

\subsection{Cosemisimplicity of the smash coproduct}
\prref{3.2} is well-known, see \cite[Cor. 5.3]{CaeDasRai}. We present a completely different proof.

\begin{proposition}\prlabel{3.2}
If $C$ and $H$ are cosemisimple, then every $M\in {\mathcal M}^{C\smashco H}$
is a direct sum in ${\mathcal M}^{C\smashco H}$ of a family of simple ${\mathcal M}^{C\smashco H}$-subcomodules. Hence every object of $\Mm^{C\smashco H}$ is semisimple, that is, 
$\Mm^{C\smashco H}$ is a semisimple category.
\end{proposition}

\begin{proof}
Assume first that $M$ is finite dimensional. It follows from \leref{2.4}(2) or \coref{2.9}(2)
(with $L=k$) that $M$ is projective in $\Mm^{C\smashco H}$. Let $N$ be
a $C \smashco H$-subcomodule $N$ of $M$. $M/N\in \Mm^{C\smashco H}$ is finite
dimensional, and therefore projective, so that the short exact sequence
$$0\rightarrow N\rightarrow M \rightarrow M/N\rightarrow 0$$
splits in $\Mm^{C\smashco H}$. This means that every $C \smashco H$-subcomodule
of $M$ is a direct summand, and $M$ is the direct sum of a finite family of
simple $C \smashco H$-subcomodules.\\
Now assume that the dimension of $M$ is possibly infinite. By the Fundamental Theorem
for comodules, every $m\in M$ is contained in a finite dimensional 
$C \smashco H$-subcomodule $V_m$ of $M$, see for example \cite[5.1.1]{Montg}
$V_m$ is the direct sum of a finite
number of simple $C \smashco H$-subcomodules of $V_m$ and, a fortiori, of $M$.
Thus every $m\in M$ is contained in a direct sum of simple $C \smashco H$-subcomodules
of $M$, and we can easily conclude that $M$ is a direct sum of a family of
simple $C \smashco H$-subcomodules.
\end{proof}

\subsection{The smash coproduct $B \bowtie H$}
Assume that $H$ is cosemisimple, and let $\phi$ be a left integral in $H^*$. Then we
have the trace map
$$\Psi:\ C\to C^{{\rm co}H},~~\Psi(c)=\phi(c_{[1]})c_{[0]}.$$
It is known that $B=C^{{\rm co}H}$ is a coalgebra with comultiplication
$$\Delta'(c) =\Psi(c_1) \otimes c_2  = c_1 \otimes \Psi(c_2) = \Psi(c_1) \otimes \Psi(c_2),$$
for all $c\in C^{{\rm co}H}$, and that $\Psi$ is a coalgebra map, see \cite[Cor. 2.4]{CaeDasRai}.
For a right $C\smashco H$-comodule $M$, we have a $k$-linear map
$$\Psi_M:\ M\to M^{{\rm co}H},~~\Psi_M(m)=\phi(m_{[1]})m_{[0]}.$$
Observe that $\Psi_C=\Psi$. $M^{{\rm co}H}$ is a right $B$-comodule, with coaction
\begin{eqnarray*}
\rho'(m)&=& m_{(0)}\ot m_{(1)}= \Psi_M(m)(m_{\{0\}}) \ot m_{\{1\}}\\
&=&m_{\{0\}} \ot \Psi(m_{\{1\}})
=\Psi_M(m)(m_{\{0\}}) \ot \Psi(m_{\{1\}}),
\end{eqnarray*}
for all $m\in M^{{\rm co}H}$. A morphism $f:\ M\to N$ in $\Mm^{C\smashco H}$
restricts and corestricts to a right $B$-colinear $f^{{\rm co}H}:\ M^{{\rm co}H}\to N^{{\rm co}H}$,
so that we have a functor $(-)^{{\rm co}H}:\ \Mm^{C\smashco H}\to \Mm^B$.\\
Now we consider $B$ as a right $H$-comodule algebra, under trivial $H$-coaction
$\rho(b)=b\ot 1$, for all $b\in B$. Then we can consider the smash coproduct $B\smashco H$,
and our previous results remain valid if we replace $C$ by $B$.

\begin{theorem}\thlabel{3.3}
Assume that $H$ is cosemisimple and take
$M \in \Mm^{B}$ and $N \in \Mm^{B \smashco H}$. $M$ is viewed as a right
 $B \smashco H$-comodule under the trivial $H$-coaction. Then
\begin{equation}\eqlabel{3.3.1}
\Ext^{B\smashco H}(M,N)^{{\rm co}H}=\Ext^B(M,N^{{\rm co}H}).
\end{equation}
\end{theorem}

\begin{proof}
We will first show that
\begin{equation}\eqlabel{3.3.2}
\Hom^{B\smashco H}(M,N)=\HOM^{B}(M,N)^{{\rm co}H}=\Hom^B(M,N^{{\rm co}H}).
\end{equation}
The first equality follows from \prref{1.1}.
Take $f\in \HOM^{B}(M,N)^{{\rm co}H}$. Then $\rho(f)=f\ot 1$, and \equref{1.2}
takes the form $f(m)\ot 1= \rho(f(m))$, for all $m\in M$, hence $f(m)\in N^{{\rm co}H}$,
and it follows that $f\in \Hom^B(M,N^{{\rm co}H})$.\\
Conversely, let $f\in \Hom^B(M,N^{{\rm co}H})$, and view $f$ as a map $M\to N$.
For all $m\in M$, $\rho(f(m))=f(m)\ot 1$, and we deduce from \equref{1.2}
that $f$ is $H$-rational and $\rho(f)=f\ot 1$, hence $f\in \HOM^{B}(M,N)^{{\rm co}H}$.
This proves the second equality.\\
Let $I$ be an injective object in $\Mm^{B\smashco H}$. In view of \equref{2.8.1},
$\EXT^{B,q}(M,I)^{{\rm co}H}\cong \Ext^{B\smashco H,q}(M, I)=0$.\\
We claim that $I^{{\rm co}H}$ is injective
in $\Mm^B$. Take an exact sequence $0\to M_3\to M_2\to M_1\to 0$ in $\Mm^B$. Viewing
the $M_i$ as $B\smashco H$-comodules via the trivial $H$-coaction, this sequence is
also exact in $\Mm^{B\smashco H}$. From the injectivity of $I\in \Mm^{B\smashco H}$,
it follows that
$$0 \rightarrow \Hom^{B \smashco H}(M_1,I)  \rightarrow \Hom^{B \smashco H}(M_2,I)  \rightarrow \Hom^{B \smashco H}(M_3,I)  \rightarrow 0$$
is exact in $\Mm_k$. In view of \equref{3.3.2}, this sequence is isomorphic to
$$0 \rightarrow \HOM^{B}(M_1,I^{{\rm co}H}) \rightarrow  \HOM^{B}(M_2,I^{{\rm co}H})  \rightarrow  \HOM^{B}(M_3,I^{{\rm co}H})\rightarrow 0.$$
It follows that  $\Ext^{B,q}(M,I^{{\rm co}H})= 0$, for all $q>0$.
Observe that functors
$$\EXT^{B}(M,-)^{{\rm co}H},~\Ext^B(M,(-)^{{\rm co}H}):\ 
\Mm^{B\smashco H}\to \Mm_k.$$
are  cohomological in the sense of \cite[2.2]{Groth}. Indeed, right derived functors are
cohomological, the coinvariants functor is exact since $H$ is cosemisimple, and it is obvious
that the composition of a cohomological functor and an exact functor is cohomological.
The two functors coincide in degree 0 \equref{3.3.2} and vanish in degree $>0$ at injective objects of 
$\Mm^{B\smashco H}$. Therefore
$ \Ext^B(M,(-)^{{\rm co}H})\cong \EXT^{B}(M,-)^{{\rm co}H}$, which is isomorphic to
$\Ext^{B\smashco H}(M, -)$ by \equref{2.8.1}.
\end{proof}

\subsection{Hochschild-Serre spectral sequence for smash coproducts}
Assume that $C$ contains a $H$-coinvariant grouplike element $x$.
This is the case for the examples (3) and (4) mentioned in the introduction.
If $H$ is cosemisimple, then $x$ is also a grouplike element of the coalgebra
$B=C^{{\rm co}H}$ discussed in the previous subsection. For a right $C$-comodule
$N$, we can consider the subspace of $C$-coinvariant elements
$$N^{{\rm co}C}=\{n\in N~|~n_{\{0\}}\ot n_{\{1\}}=n\ot x\}.$$

\begin{lemma}\lelabel{3.4}
If $N\in \Mm^{C\smashco H}$, then $N^{{\rm co}C}\in \Mm^H$.
\end{lemma}

\begin{proof}
Fix a basis $\{h_i~|~i\in I\}$ of $H$. Take $n\in N^{{\rm co}C}$ and write $\rho^H(n)=
n_{[0]}\ot n_{[1]}= \sum_i n_i\ot h_i$. Then
\begin{eqnarray*}
&&\hspace*{-2cm}
\sum_i \rho^C(n_i)\ot h_i=\rho^C(n_{[0]})\ot n_{[1]}\equal{\equref{star}}
n_{\{0\}[0]} \otimes n_{\{1\}[0]} \otimes n_{\{0\}[1]}n_{\{1\}[1]}\\
&=& n_{[0]} \otimes x_{[0]} \otimes n_{[1]}x_{[1]}
=n_{[0]} \otimes x \otimes n_{[1]}= \sum_i n_i\otimes x\ot h_i
\end{eqnarray*}
It follows that each $n_i$ is $C$-coinvariant, and $\rho(n)\in N^{{\rm co}C}\ot H$.
\end{proof}

Observe that $x\smashco 1_H\in G(C\smashco H)$. $k$ is a right $C\smashco H$-comodule
with coaction $\rho(\lambda)=\lambda x\smashco 1_H$, and
$N^{{\rm co}C}\cong \Hom^C(k,N)$. Consequently we have isomorphisms of derived
functors $\Ext^C(k,-)\cong R^p(-)^{{\rm co}C}$. Taking $M=k$ in \coref{2.7}, we
obtain the following spectral sequence.

\begin{theorem}\thlabel{3.5}
Assume that $C$ contains a $H$-coinvariant grouplike element $x$.
For $N\in \Mm^{C\smashco H}$, we have a spectral sequence
$$R^p(R^q(N)^{{\rm co}C})^{{\rm co}H} \Rightarrow R^{p+q}(N)^{{\rm co}C \smashco H}.$$
\end{theorem}

Applying \thref{3.5} to Examples (3) and (4) from the introduction, we obtain the following results.

\begin{corollary}\colabel{3.6}
Let $A$ be a Hopf algebra, and let $G$ be a finite group acting as a group of Hopf algebra
automorphisms on $A$. For $N\in \Mm^{A\smashco k^G}$, we have a spectral sequence
$$R^p(R^q(N)^{{\rm co}A})^{{\rm co}k^G} \Rightarrow R^{p+q}(N)^{{\rm co}A\smashco k^G}.$$
\end{corollary}

\begin{corollary}\colabel{3.7}
Assume that an affine algebraic group $G$ can be written as the semi-direct product
of two algebraic subgroups $H$ and $K$. Let $A(G)$ be the coordinate ring of $G$.
For $N\in \Mm^{A(G)}$, we have a spectral sequence
$$R^p(R^q(N)^{{\rm co}A(L)})^{{\rm co}A(K)} \Rightarrow R^{p+q}(N)^{{\rm co}A(G)}.$$
\end{corollary}

We remark that \coref{3.7} can be reformulated as follows: for a rational $G$-module $N$,
we have a spectral sequence
$$H^{p}(K, H^{q}(L,N)) \Rightarrow H^{p+q} (G,N).$$
Here $H^p(G,-)=R^p(-)^H$ is the $p$-th right derived functor of the $G$-invariants functor
$(-)^G$, from rational $G$-modules to vector spaces. Indeed, the category of rational
$G$-modules is isomorphic to the category of $A(G)$-comodules, and $G$-invariants
of a rational $G$-module are precisely the $G$-coinvariants of an $A(G)$-comodule.
For more information on this spectral sequence, see 
\cite[Theorem 2.9]{Haboush} and \cite[Lemma 1.1]{ClineParsh}.

\end{document}